\documentclass[12pt]{article}

\usepackage[english]{babel}
\usepackage{amsmath,amssymb,amsfonts,amscd}
\usepackage{amsthm}

\newcommand{\cc}{\mathcal}
\newcommand{\bb}{\mathbb}
\newcommand{\bacc}{\left\{}
\newcommand{\eacc}{\right\}}

\newcommand{\bp}{\left(}
\newcommand{\ep}{\right)}
\newcommand{\bint}{\left[}
\newcommand{\eint}{\right]}
\newcommand{\rn}{$\mathbb{R}^N$}
\newcommand{\sla}{\backslash}

\newcommand{\bdisp}{\begin{displaymath}}
\newcommand{\edisp}{\begin{displaymath}}
\newcommand{\bsplit}{\begin{split}}
\newcommand{\esplit}{\end{split}}

\newtheorem{defi}{Definition}[section]
\newtheorem{lem}[defi]{Lemma}
\newtheorem{pro}[defi]{Proposition}
\newtheorem{thm}[defi]{Theorem}
\newtheorem{cor}[defi]{Corollary}

\usepackage{a4}

\makeatletter
\@addtoreset{equation}{section}
\makeatother

\begin{document}
\title{Comparison of harmonic kernels associated to a class of semilinear elliptic
equations}
\author{Mahmoud~Ben~Fredj \\ {\normalsize\em Facult\'e des Sciences de Monastir
}\\{\normalsize\em Avenue de l'Environnement, 500 Monastir,
Tunisia}\\
{\small E-mail: mahmoudbenfredj@yahoo.fr}\bigskip\\
Khalifa~El~Mabrouk\\ {\normalsize\em D\'epartement de
Math\'ematique}
\\{\normalsize\em Ecole Sup\'erieure des Sciences et de Technologie de Hammam Sousse}
\\{\normalsize\em Rue Lamine El Abbassi, 4011 Hammam Sousse, Tunisia}
\\{\small E-mail: khalifa.elmabrouk@fsm.rnu.tn} }
\date{}
\maketitle

\abstract{Let $D$ be a smooth domain in $\mathbb{R}^N$, $N\geq 3$ and let $f$
be a positive continuous function
 on $\partial D$. Under some assumptions on  $\varphi$, it is shown that the   problem $\Delta u=2\varphi(u)$
  in~$D$ and $u=f$ on $\partial D$,  admits a unique solution which will  be denoted by $H_D^\varphi f$.
   Given two functions $\varphi$ and $\psi$, our main goal in this paper is to investigate the existence of a constant $c>0$ such that
$$\frac{1}{c}H_D^\varphi f\leq H_D^\psi f\leq c H_D^\varphi f.$$
}

\section{Introduction}

Let $D$ be a bounded smooth domain in $\bb{R}^N,N\geq 3.$ We
consider the following semilinear problem
\begin{equation}\label{e31}
\left\{\begin{array}{rlll}\Delta u&=&2\varphi(u)& \mbox{in}\ D,\\
u&=&f&\mbox{on}\ \partial D,\end{array}\right.
\end{equation}
where $f$ is a positive continuous function on~$\partial D.$ Under
some conditions on~$\varphi,$ it will be shown that
problem~(\ref{e31}) admits a unique solution which will be denoted
by~$H_D^\varphi f.$ In the particular case where
 $\varphi\equiv 0,$ (\ref{e31}) reduces to the classical Dirichlet problem whose the unique solution will
 be denoted by $H_Df.$

 Given two functions $\varphi$ and $\psi,$ we say that $H_D^\varphi f$ and
 $H_D^\psi f $ are proportional and we write $H_D^\varphi f \approx H_D^\psi f$
 if there exists  $c>0$ such that for every $x\in D$,
 $$
 \frac{1}{c}H_D^\psi f (x)\leq H_D^\varphi f (x)\leq c H_D^\psi f (x).
 $$
The operators $H_D^\varphi $ and
 $H_D^\psi $ are said to be proportional (we write $H_D^\varphi \approx H_D^\psi $) if $H_D^\varphi f$ and
 $H_D^\psi f $ are proportional for every positive continuous
 function~$f$ on~$\partial D$.

The main goal of this paper is to study the
 proportionality between  $H_D^\varphi$ and  $H_D^\psi$. To this end,  we shall  rather compare $H_D^\varphi$ to $H_D$. Since
$f$ is positive on~$\partial D$, it is very simple to observe that
$H_D^\varphi f\leq h$ where $h$ is the positive harmonic function
$H_D f$. Hence, the key question is whether $H_D^\varphi f\geq c h$
for some positive constant~$c$.

 There are several papers dealing with the existence of  solutions to semilinear problems which are bounded below by a harmonic function
   (see \cite{2,3,4,8} and their references). The second author~\cite{8} studied the problem
\begin{equation}\label{e41}
\left\{\begin{array}{rlll}\Delta u+\xi(x)\Psi(u)&=&0& \mbox{in}\ D,\\
u&>&h&\mbox{in}\  D,\\
u-h&=&0& \mbox{on}\ \partial D,\end{array}\right.
\end{equation}
 where $h\geq 0$ is harmonic in $D,$ $\xi\geq 0$ is locally bounded and $\Psi>0$ is a nonincreasing continuous
 function on  $]0,\infty[$. He proved  that (\ref{e41}) admits a unique solution
  provided the function
$$
x\mapsto \int_D G_D(x,y)\xi(y)\,dy
$$
 is continuous on $D$ and vanishes on $\partial D$, where
 $G_D(\cdot,\cdot)$ denotes the Green function of~$\Delta$ on~$D$ (see (\ref{gf})
 below).

Athreya \cite{3} considered the problem (\ref{e31})  where
$\varphi:\bb{R}_+\rightarrow \bb{R}_+$ is
   locally H\"older continuous and decays to~$0$ at the same rate as  $ t^p,0<p<1.$
    Given a  function $h_0$ which is continuous on $\overline{D}$ and harmonic in~$D,$  he showed
    that there exists $c>1$ such that for every continuous function~$f$ on~$\partial D$ satisfying
$$
f\geq c h_0\quad\mbox{on }\partial D,
$$
 problem (\ref{e31}) has a unique solution which is bounded below by~$h_0.$ By probabilistic techniques,
   Chen, Williams and Zhao investigated in~\cite{4} the same  problem  where  $-t\leq\varphi(t)\leq t.$
    They proved the existence of a solution bounded below by a positive harmonic function provided the norm of $f$ is
       sufficiently small.

The problem (\ref{e31}), with $\varphi (t)=t^p$, was already studied
by Atar, Athreya and Chen in \cite{2}. They  showed that the
proportionality of   $H_D^\varphi$ and $H_D$ holds true if $p\geq
1.$  In this direction we shall prove in this paper that
$H_D^\varphi\approx H_D$ for a large class of functions $\varphi$.
 On the other hand, again in \cite{2}, it was conjectured  that $H_D^\varphi$ and $H_D$ are not proportional when
  $\varphi$ is given by  $\varphi (t)=t^p,$ $0\leq p<1.$ Here, we shall prove this conjecture. More precisely,  we
   give a sufficient condition on $\varphi$ under which the proportionality does not hold.

After recalling in the following section some basic facts on
Brownian motion, we establish in Section~3 the existence of a unique
solution to problem~(\ref{e31}) where $\varphi:\bb{R}_+\to
\bb{R}_+$ is continuous nondecreasing and satisfies $\varphi(0)=0$.

In Section~4 we are concerned with the proportionality between
$H_D^\varphi$ and the harmonic kernel  $H_D$. We prove that the
proportionality holds true provided
\begin{equation}\label{elimsup}
\limsup_{t\rightarrow0}\frac{\varphi(t)}{t}<\infty,
\end{equation}
and does not hold if for some $\varepsilon>0,$
\begin{equation}\label{e241}
 \int_{0}^{\varepsilon}\bp \int_{0}^{t}\varphi(s)ds\ep^{-\frac{1}{2}}dt<\infty.
\end{equation}

 Seeing that condition (\ref{e241}) is valid  for $\varphi(t)=t^p$ with $0\leq p<1,$ the second part of our above result  gives an
  immediate proof of the conjecture  mentioned above.

  The last  section  will be devoted to investigate problem~(\ref{e31}) in the
  case where the function $\varphi$ is nonincreasing.

\section{Preliminaries}

 For every subset  $F$ of $\bb{R}^N$, let $\cc{B}(F)$ be the set of all Borel measurable functions on $F$ and
 let $\cc{C}(F)$ be the set of all continuous real-valued functions on $F.$ If $\cc{G}$ is a set of numerical
  functions then $\cc{G}^+$ (respectively $\cc{G}_b$) will denote the class of all functions in $\cc{G}$ which
  are nonnegative (respectively bounded). The uniform convergence norm will be denoted by $\left\|.\right\|.$

Let $(\Omega,\cc{F},\cc{F}_t,X_t,P^x)$ be the canonical Brownian
motion on the Euclidian space \rn, $N\geq 3$: $\Omega $ is the set
of all continuous functions from $\left[0,\infty\right[$ to~\rn
endowed with its Borel $\sigma$-algebra $\cc{F}.$ For every $t\geq
0$ and $\omega\in\Omega,$
$$
X_t(\omega)=\omega(t)\quad\mbox{and}\quad \cc{F}_t:=\sigma(X_s;0\leq
s\leq t).
$$
Moreover, for every $x\in \bb{R}^N,$ $P^x$ is the probability
measure on $(\Omega,\cc{F})$ under which the Brownian motion starts
at $x$ (i.e., $P^x(X_0=x)=1$) and $E^x[\cdot]$  denotes the
corresponding expectation. Let
 $D$ be a bounded domain in \rn and let   $\tau_D$ be  the first exit time from $D$ by $X,$
 i.e.,
$$
\tau_D=\inf\left\{t>0;X_t\notin D\right\}.
$$
Let us denote by $(X_t^D)$ the Brownian motion killed upon exiting
$D.$ It is well known that the transition density is given by
$$
p^D(t,x,y)=p(t,x,y)-r^D(t,x,y);\quad \ t>0,\ x,y\in D,
$$
\begin{displaymath}
\begin{split}
\mbox{where}\qquad&p(t,x,y)=\frac{1}{(2\pi
t)^{N/2}}\exp\left({-\frac{\left|x-y\right|^2}{2
t}}\right)\\
\mbox{and}\qquad
&r^D(t,x,y)=E^x\left[ p(t-\tau_D,X_{\tau_D},y),\tau_D< t\right].
\end{split}
\end{displaymath}
The corresponding semigroup is then defined by
$$P^D_tf(x)=E^x\left[f(X_t) , t<\tau_D \right]=\int_D p^D(t,x,y)f(y)dy,\quad x\in D,$$
for every Borel measurable function $f$ for which this integral
makes  sense.

Let $h$ be a positive harmonic function in $D.$ We define for
$x,y\in D,t>0$,
$$
p_h^D(t,x,y)=p^D(t,x,y)\frac{h(y)}{h(x)}.
$$
  There exists a Markov process, called the $h$-conditioned Brownian motion, with state space $D$
   and having $p_h^D$ as  transition density  (see \cite{12,5,6}).  The corresponding probability
    measures is denoted by  $(P^x_h)_{x\in D}:$ for every Borel subset $B$ of $D$ we
    have
\begin{displaymath}
\begin{split}
 P_h^x(X_t\in B)&=\frac{1}{h(x)}\int_B p^D(t,x,y) h(y)\,dy\\
 &=\frac{1}{h(x)}E^x\left[h(X_t), X_t\in B,t<\tau_D \right].
\end{split}
\end{displaymath}
Besides, using the monotone class theorem, it is easily seen that
for every $t>0$ and   every  $\cc{F}_t$-measurable randan  variable
$Z\geq 0$,
\begin{equation}
E^x_h\left[Z,t<\tau_D\right]=\frac{1}{h(x)}E^x\left[Z\,
h(X_t),t<\tau_D\right].
\end{equation}

 The open bounded subset $D$  is called regular (for $\Delta$) if each function $f\in\cc{C}(\partial D)$
  admits a continuous extension $H_Df$ on $\overline{D}$ such that $H_Df$ is harmonic in $D.$ In other words,
  the function $h=H_Df$ is the unique solution to the classical Dirichlet problem
$$
\left\{
\begin{array}{rrrl}
\Delta h&=&0& \textrm{in}\  D,\\
h&=&f&\textrm{on}\ \partial D.
\end{array}
\right.
$$
For every $x\in D,$ the harmonic measure relative to $x$ and $D,$
which will be denoted by $H_D(x,\cdot),$
 is defined to be the positive Radon measure on $\partial D$ given by the mapping $f\mapsto H_Df(x).$

In the sequel, we always assume that $D$ is regular and let $x_0\in
D$ be a fixed point. There exists a unique function $K_D:D\times
\partial D\rightarrow \bb{R}_+$ satisfying:
\begin{enumerate}
\item[] For every $z\in \partial D,$ $K_D(x_0,z)=1.$
\item[] For every $x\in D,$ $K_D(x,\cdot )$ is continuous in  $\partial D.$
\item[] For every $z\in \partial D,$ $K_D(\cdot ,z)$ is a positive harmonic function in  $D.$
\item[] For every $z,w\in \partial D$ such that  $z\neq w,$  $\lim_{x\rightarrow w}K_D(x,z)=0.$
\end{enumerate}
We  extend the function $K_D(\cdot ,z)$ to
$\overline{D}\sla\left\{z\right\}$ by letting
  $K_D(w,z)=0$ for every  $w\in \partial D\sla \left\{z\right\}.$ The function $K_D$  is called the Martin kernel on   $D.$

The Green function $G_D(\cdot,\cdot)$ is defined on $D\times D$ by
\begin{equation}\label{gf}
G_D(x,y)=\int_{0}^{\infty}p^D(t,x,y)dt.
\end{equation}
It is well known that $G_D$ is continuous (in the extended sense) on
$D\times D,$
 \begin{displaymath}
  G_D(x,y)\leq G_{\mathbb{R}^N}(x,y)=\frac{\Gamma(\frac{N}{2}+1)}{2\pi^{\frac{N}{2}} |x-y|^{N-2}},
 \end{displaymath}
  and $\lim_{x\rightarrow z}G_D(x,y)=0$
  for every $z\in\partial D$ (see \cite[chapter 4]{10}). Moreover
\begin{equation}
K_D(x,z)=\frac{dH_D(x,\cdot)}{d H_D(x_0,\cdot)}(z)=\lim_{y\in D,
y\rightarrow z}\frac{G_D(x,y)}{G_D(x_0,y)};\quad \ x\in
D,z\in\partial D.
\end{equation}

  For $h=K_D(\cdot,z)$ where $z\in\partial D$, the $h$-conditioned Brownian motion will be simply called
   the $z$-Brownian motion and its transition density is given by
 $$p^D_z(t,x,y)=\frac{1}{K_D(x,z)}p^D(t,x,y)K_D(y,z); \ \ t>0,x,y\in D.$$
The corresponding probability measures family will be denoted by
 $(P^x_z)_{x\in D}.$

\section{Semilinear problem}

In the sequel, we assume that  $\varphi:\bb{R}_+\rightarrow
\bb{R}_+$ is a continuous nondecreasing function such that
$\varphi(0)=0.$ The following  comparison principle  will be useful
to prove not only uniqueness but also the existence
 of a solution to  problem~(\ref{e31}). A more general comparison principle can be found in \cite{8}.

\begin{lem}\label{cprin}
Let $\Psi\in\cc{B}(\bb{R})$ be a nondecreasing function and let
$u,v\in \cc{C}(\overline{D})$ such that
$$
\Delta u\leq \Psi(u)\qquad \mbox{and} \qquad \Delta v\geq \Psi(v)\qquad \mbox{in}\ D.
$$
If $u\geq v$ on $\partial D,$ then $u\geq v$ in $D.$
\end{lem}

\begin{proof}  Define $w=u-v$ and suppose that the open set
$$
\Omega=\left\{x\in D;w(x)<0\right\}
$$
 is not empty. Since $\Psi$ is nondecreasing, it is obvious that  $\Delta w\leq \Psi(u)-\Psi(v)\leq 0$
 in~$\Omega$, which means that  $w$ is superharmonic in~$\Omega.$ Furthermore,
  for every $z\in\partial\Omega\cap D$ we have $w(z)=0$ (because $w$ is continuous in~$D$),
  and for every  $z\in \partial \Omega\cap \partial D$ we have $\lim_{x\in \Omega,x\rightarrow z}w(x)\geq 0$
   (by hypothesis). Then  $w\geq 0$ in $\Omega$  by the  classical minimum principle for superharmonic
   functions. This yields a contradiction and therefore~$\Omega$ is
   empty. Hence $u\geq v$ in $D.$
\end{proof}

  The Green operator in $D$ is defined, for every Borel measurable function~$f$ for which the following
   integral exists,  by
  \begin{equation}\label{eog}
  G_Df(x)=\int_DG_D(x,y)f(y)dy,\qquad x\in D.
  \end{equation}
   Hence
  $$
  G_Df(x)=E^x\left[\int_{0}^{\tau_D}f(X_t)dt\right]=\int_{0}^{\infty}P_t^Df(x)dt,\quad x\in D.
  $$
We  recall that for every  $f\in \cc{B}_b(D)$,
 $G_Df$ is a bounded continuous function on~$D$ satisfying $\lim_{x\rightarrow z}G_Df(x)=~0$ for every  $z\in \partial D.$
 Moreover, it is simple to check that
 $$\Delta G_Df=-2f$$ in the distributional sense (see \cite{5,6}).

\begin{lem}\label{l221}
For every $M>0,$ the family  $\left\{G_Du;\left\|u\right\|\leq
M\right\}$ is relatively compact with respect to the uniform
convergence norm.
\end{lem}
\begin{proof}
First, we recall that $x\mapsto G_D1(x)=E^x[\tau_D]$ is bounded
on~$D$ (see, e.g.,
\cite[page 23]{10}) and consequently for every $u$ such that $\left\|u\right\|\leq~
M$ we get
$$
\left\|G_Du\right\|\leq M \sup_{x\in D}E^x[\tau_D].
$$
Thus the family $\left\{G_Du;\left\|u\right\|\leq M\right\}$ is
uniformly bounded. Next, we claim  that the family
 $\left\{G_D(x,\cdot); x\in D\right\}$ is uniformly integrable.
 Indeed,
let $\varepsilon>0$ and   $\eta_0>0.$ There exist $c_1>0$ and
$c_2>0$ such that for every Borel subset $A$ of $D$,
\begin{displaymath}
\begin{split}
\int_A G_D(x,y)dy&\leq c_1 \int_A \frac{dy}{\left|x-y\right|^{N-2}}\\
&\leq
c_1\int_{B(x,\eta_0)}\frac{dy}{\left|x-y\right|^{N-2}}+c_1\int_{A\sla
B(x,\eta_0)}
\frac{dy}{\eta_0^{N-2}}\\
&\leq c_2\eta_0^2+c_2\frac{m(A)}{\eta_0^{N-2}}.
\end{split}
\end{displaymath}
Here and in all the following,  $m$ denotes the Lebesgue measure in
\rn. Take $\eta_0= \sqrt{\varepsilon/2c_2}$ and
$\eta=\varepsilon\eta_0^{N-2}/2c_2$.
 Then  for every  Borel subset $A$ of $D$  such that
$m(A)<\eta$ we have
$$
\int_A G_D(x,y)\, dy\leq \varepsilon.
$$
Hence, the uniform integrability of the family $\left\{G_D(x,\cdot);
x\in D\right\}$ is shown. Therefore, in virtue of Vitali's
convergence theorem (see, e.g; \cite{rud}),  we conclude that for
every $z\in D$,
\begin{displaymath}
\begin{split}
\lim_{x\to z}\sup_{\|u\|\leq M}\left|\int_D G_D(x,y)u(y)dy-\int_DG_D(z,y)u(y)dy\right|\\
\qquad\qquad\leq M \lim_{x\to
z}\int_D\left|G_D(x,y)-G_D(z,y)\right|dy= 0.
\end{split}
\end{displaymath}
This means that   the family $\left\{G_D(x,\cdot); x\in D\right\}$
is equicontinuous  which finishes the proof of the lemma.
\end{proof}

Existence of solutions to semilinear Dirichlet problems of
kind~(\ref{e31}) was widely studied in the literature considering
 various hypotheses on the function~$\varphi$ (see, e.g.,
 \cite{baha02,dyn00,elm,8}). In our setting, we get the following
 theorem.

\begin{thm}\label{thm81}
For every $f\in \cc{C}^+(\partial D),$ there exits one and only one
function $u\in\cc{C}^+(\overline{D})$ satisfying
problem~(\ref{e31}). Furthermore, a bounded Borel function~$u$
on~$D$ is a solution to~(\ref{e31}) if and only if
$u+G_D\varphi(u)=H_Df.$
\end{thm}
\begin{proof} By a classical computation, it is not hard to
establish the second part of the theorem. We also observe that, by
the comparison principle (Lemma~\ref{cprin}),  problem~(\ref{e31})
possesses at most one solution. So, it remains to prove the
existence of a solution to~(\ref{e31}). Take $f\in \cc{C}^+(\partial
D),$ $a=\left\|f\right\|$, $M=a+\varphi(a)\sup_{x\in D}E^x[\tau_D]$
and define $
\Lambda=\left\{u\in
\cc{C}(\overline{D});\left\|u\right\|\leq M\right\}.$
 Let $h=H_Df$ and consider the operator
$T:\Lambda\rightarrow
\cc{C}(\overline{D})$  defined by
$$Tu(x)=h(x)-E^x\left[\int_{0}^{\tau_D}g(u(X_s))ds\right],\quad \ x\in D,$$
where~$g$ is the real-valued odd function given by
$g(t)=\inf(\varphi(t),\varphi(a))$ for every   $t\geq 0$. Since
$\left|g(t)\right|\leq
\varphi(a)$ for every $t\in\bb{R},$ we get
$$\left|Tu(x)\right|\leq M$$
for every $x\in D$ and every $u\in\Lambda.$ This implies that
$T(\Lambda)\subset \Lambda.$ Now, let  $(u_n)_{n\geq 0}$ be a
sequence in $\Lambda$ converging uniformly to
  $u\in\Lambda.$ Let $\varepsilon >0.$ Since $g$ is uniformly continuous in
  $[-M,M],$ we deduce that there exists
    $n_0\in\mathbb{N}$ such that for every $n\geq n_0$ and $s\in[0,\tau_D]$
$$\left|g(u_n(X_s))-g(u(X_s))\right|<\varepsilon.$$
It follows that for every  $n\geq n_0$ and  $x\in D$,
\begin{displaymath}
\begin{split}
\left|Tu_n(x)-Tu(x)\right|&=\left|E^x\left[\int_{0}^{\tau_D}g(u_n(X_s))ds\right]
-E^x\left[\int_{0}^{\tau_D}g(u(X_s))ds\right]\right|\\
&\leq E^x\left[\int_{0}^{\tau_D}\left|g(u_n(X_s))-g(u(X_s))\right|ds\right]\\
&\leq \varepsilon \sup_{x\in D}E^x[\tau_D].
\end{split}
\end{displaymath}
This shows that  $(Tu_n)_{n\geq 0}$ converges uniformly to $Tu.$ We
then conclude that $T$ is a continuous operator. On the other hand,
$\Lambda$ is a closed bounded convex subset of
$\cc{C}(\overline{D}).$ Moreover, in virtue of Lemma~\ref{l221},
$T(\Lambda)$ is relatively compact. Thus, the Schauder's fixed point
theorem ensures  the existence of  a function  $u\in \Lambda$ such
that $u=h-G_Dg(u).$ Applying the comparison principle, we obtain
that  $0\leq u\leq a$  and so  $g(u)= \varphi(u).$ Hence, the proof
is finished.
\end{proof}

The unique solution to  problem (\ref{e31}) will be always   denoted
by $H_D^\varphi f.$ However, in the particular  case where
$\varphi:t\mapsto t^p,p>0,$ we may   write  $H_D^pf$ instead of
$H_D^\varphi f.$

\section{Proportionality  of $H_D^\varphi f$ and $H_Df$}

In the sequel, we suppose that $D$ is a Lipschitz bounded domain of
$\mathbb{R}^N$. We recall  the Feynman-Kac theorem  (see
\cite[Theorem 4.7]{5}) which plays an important role in what
follows: for every $f\in
\cc{C}^+(\partial D)$ and  $q \in\cc{B}_b^+(D),$  the function
$v\in\cc{C}(\overline{D})$ given by
\begin{equation}\label{e51}
v(x)=E^x\left[f(X_{\tau_D})\exp\bp-\int_{0}^{\tau_D}q(X_s)\,ds\ep\right],\quad
x\in D,
\end{equation}
 is the unique solution of the problem
 $$\left\{\begin{array}{rrll}\Delta v&=&2q v&\textrm{in } D,\\
 v&=&f&\textrm{on } \partial D.
 \end{array}\right.$$
 Let us notice  that $v$ given by (\ref{e51}) satisfies the following integral equation:
$$v(x)=h(x)-\int_D G_D(x,y)q(y)v(y)dy,\quad x\in D.$$

 Our first result in this section is the following:
\begin{thm}\label{t7}
Assume that
\begin{equation}\label{lim}
\limsup_{t\rightarrow 0}\frac{\varphi(t)}{t}<\infty.
\end{equation}
Then $H_D^\varphi f \approx H_Df$ for every function
$f\in\cc{C}^+(\partial D).$
\end{thm}
\begin{proof} Let $f\in\cc{C}^+(\partial D)$ be nontrivial, that is $h=H_Df>0$ in $D.$ Let $u=H_D^\varphi f$
and define  $$q:=\frac{\varphi(u)}{u}1_{\{u>0\}}.$$ Then $q$ is a
positive bounded function
 in $D$  by
(\ref{lim}), and $u$ satisfies the problem
\begin{equation}
\left\{\begin{array}{rlll}\Delta u&=&2q\,u& \mbox{in}\ D,\\
u&=&f&\mbox{on}\ \partial D.\end{array}\right.
\end{equation}
 We define
$$
w(x,z)=E_z^x\bint\exp\bp-\int_{0}^{\tau_D}q(X_t)dt\ep\eint,\qquad
x\in D,z\in\partial D.
$$
By  Feynman-Kac theorem and \cite[Proposition 5.12]{5}, we have
\begin{eqnarray}
u(x)&=&E^x\left[f(X_{\tau_D})\exp\bp-\int_{0}^{\tau_D}q(X_s)\,ds\ep\right]\nonumber\\
&=&\int_{\partial D} w(x,z)f(z)H_D(x,dz).\label{e213}
\end{eqnarray}
Since for every $x\in D$
   $$ E^x\bint \exp\bp-\int_{0}^{\tau_D}q(X_s)\,ds\ep\eint<\infty,$$
     by  \cite[Theorem 7.6]{5} there exists $c>0$ such that
\begin{equation}\label{e212}
\frac{1}{c}\leq w(x,z)\leq c,\qquad x\in D,z\in\partial D.
\end{equation}
Combining (\ref{e213})  and (\ref{e212}) we conclude that for every
$x\in D,$
\begin{displaymath}
\begin{split}
\frac{1}{c}H_Df(x)&=\frac{1}{c}\int_{\partial D}f(z)H_D\left(x,dz\right)\\
&\leq u(x)\\
&\leq c\int_{\partial D}f(z)H_D\left(x,dz\right)=cH_Df(x).
\end{split}
\end{displaymath}
Hence, $H_Df\approx H_D^\varphi f.$
\end{proof}

Let us notice that the hypothesis mentioned in the previous theorem
will be trivially  satisfied provided  the function $t\mapsto
\varphi(t)/t$ is nondecreasing or if it is bounded and
nonincreasing on $]0,\infty[$. In particular, it follows that
$H_D^\varphi f\approx H_Df$ for every function
$f\in\cc{C}^+(\partial D)$ if the function $\varphi$ is given
$$
\varphi(t)=t^p\quad\mbox{with }p\geq 1\quad\mbox{or}\quad\varphi(t)=\log
(1+t).
$$

 We shall write $H_D^\varphi\approx H_D$ if $H_D^\varphi f\approx H_Df$ for every
function $f\in\cc{C}^+(\partial D).$ Hence, by Theorem~\ref{t7},
 $H_D^p\approx H_D$ for every $p\geq 1$. This was  established
 by Atar, Athreya and Chen in~\cite{2}. In the same paper,
 the authors conjectured that $H_D^p\not\approx H_D$ for every $
 0<p<1$.  In the following, we
shall prove this conjecture. More precisely, we  give a sufficient
condition on~$\varphi$ under which $H_D^\varphi\not\approx H_D.$

\begin{thm}\label{nonpro}
 Assume that there exists $\varepsilon >0$ such that
\begin{equation}\label{e201}
\int_{0}^{\varepsilon}\bp \int_{0}^{s}\varphi(r)\,dr\ep^{-\frac{1}{2}}ds<\infty.
\end{equation}
 Then there exists
 $f\in\cc{C}^+(\partial D)$ such that $H_D^\varphi f\not\approx H_Df.$
\end{thm}

\begin{proof} We easily observe, in virtue of condition (\ref{e201}), that the function $Q$
 defined for every $t\geq 0$ by
$$
Q(t)=\frac{1}{2}\int_{0}^{t}\bp\int_{0}^{s}\varphi(r)\,dr\ep^{-\frac{1}{2}}ds
$$
 is  increasing continuous on
$[0,\infty[$,  twice differentiable on $]0,\infty[$ and   invertible
from $[0,\infty[$ to $[0,\bar\rho[$ where
$\bar\rho:=\lim_{t\to\infty} Q(t)$ (notice that $\bar\rho$ may be
infinite). Without loss of generality, we assume that
$$
D\subset \bb{R}^N_+\cap B(0,\rho)
\quad\mbox{ and }\quad 0\in\partial D,
$$
where $\bb{R}^N_+=\{x=(x_1,\cdots,x_N)\in\bb{R}^N;x_1> 0\}$ and
$0<\rho<\bar\rho$. Consider the function $u$ defined by
$u(x)=R(x_1)$ for every $x=(x_1,\dots,x_n)\in
\overline{D}$ where $R$ denotes  the inverse function of~$Q$. Then,
it is obvious  that
$$
u\in\mathcal{C}(\overline{D})\cap\mathcal{C}^2(D).
$$
Moreover, an elementary calculus yields  that for every $x\in D$,
$$
\Delta u(x)=R^{''}(x_1)=2\varphi(R(x_1))=2\varphi(u(x)).
$$
Hence $u=H_D^\varphi f$ where $f=u_{|\partial D}.$ Let $h=H_D f$ and
consider the harmonic function $g:x\mapsto x_1$. From the boundary
Harnack principle it follows that there exists an open neighborhood
$V$ of $0$ such that
$$
g\approx h \quad\mbox{in }  V\cap D.
$$
On the other hand
$$
\lim_{x\rightarrow 0}\frac{u(x)}{g(x)}=
\lim_{t\to 0}\frac{t}{Q(t)}=0.
$$
Thus $u\not\approx g$ and consequently $u\not\approx h.$
\end{proof}
Since  the function $\varphi:t\mapsto t^p$   satisfies (\ref{e201})
for $0<p<1$, we deduce from the previous theorem that, for small
$p$, $H_D^p\not\approx H_D$ which proves the conjecture given
in~\cite{2}.

In the remainder of this section, we shall proceed to answer the
following question: in the case where (\ref{lim}) fails, for which
function $f\in\cc{C}^+(\partial D),$ the proportionality of $H_Df$
and $H_D^\varphi f$ does hold?

First, the following proposition is easily obtained.
\begin{pro}\label{cspro}
Let  $f\in\cc{C}^+(\partial D)$, $h=H_Df$ and $u=H_D^\varphi f.$ If
\begin{equation}\label{p}
\sup_{x\in D} \frac{1}{h(x)}\int_D G_D(x,y)\varphi(h(y))\,dy<1,
\end{equation}
then $u\approx h$.
\end{pro}
\begin{proof}
It is an immediate consequence of the formula $h=u+G_D\varphi(u)$
and the fact that $\varphi$ is nondecreasing.
\end{proof}

Hence, one direction in solving  the question above is to
investigate functions~$f$ for which  condition~(\ref{p}) is
fulfilled. Let us
 notice that "$<1$" in~(\ref{p}) can not be replaced by "$<\infty$".
 In fact, as will be shown below, for smooth domain~$D$ we always have
\begin{equation}\label{cond}
\sup_{x\in D} \frac{1}{h(x)}\int_D G_D(x,y)\varphi(h(y))\,dy<\infty.
\end{equation}
However, for $\varphi(t)=t^p$ with $0<p<1$, Theorem~\ref{nonpro}
proves that there exists a function $f\in\cc{C}^+(\partial D)$ such
that $u$ and $h$ are not proportional.

From now on, we assume that $D$ is a bounded $C^{1,1}$ domain. Let
$\delta(x):=\inf_{z\in\partial D}|x-z|$  be the Euclidean distance
from $x\in D$ to the boundary of~$D$. We recall that
Zhao~\cite{zha86} established  the following:
   \begin{equation}\label{e122}
  G_D(x,y)\approx \min \bacc\frac{1}{|x-y|^{N-2}},\frac{\delta(x)\delta(y)}{|x-y|^N}\eacc.
  \end{equation}

\begin{lem}\label{lsupfini}
 For every positive harmonic function $h$ in $D$, there exists a
 positive constant $c$ such that for every $x\in D$,
 $$
\int_D G_D(x,y)\,dy\leq c \,h(x).
 $$
 \end{lem}

\begin{proof} Let $h$ be a positive harmonic function in $D$.
 Since the Euclidian boundary of~$D$ coincides
with the Martin one, there exists a positive Borel measure~$\nu$ on
$\partial D$ such that
\begin{equation}\label{mar}
h=\int_{\partial D}K_D(\cdot,z)\,d\nu(z).
\end{equation}
We claim that there exits $C>0$ such that for every $x\in D$ and
$z\in\partial D$,
\begin{equation}\label{mar1}
 h(x)\geq C \delta (x).
\end{equation}
Indeed, let $x\in D$ and $z\in \partial D$. Then, it is simple to
observe that $\delta(x)\delta (y)\leq |x-y|^2$ for every $y\in D$
such that $8|y-z|<\delta(x)$. Hence, by~(\ref{e122}) there exists a
constant $c_1>0$ such that for every $y\in D\cap B(z,\delta(x)/8)$,
$$
 G_D(x,y)\geq c_1\frac{\delta(x)\delta(y)}{\left|x-y\right|^N}.
 $$
Again by~(\ref{e122}) there exists  $c_2>0$ such that
 $$
 G_D(x_0,y)\leq c_2\frac{\delta(x_0)\delta(y)}{|x_0-y|^N},
 $$
where $x_0$ denotes, as was mentioned in Section~2, a reference
point. Therefore for every  $y\in D\cap B(z,\delta(x)/8)$,
$$
\frac{G_D(x,y)}{G_D(x_0,y)}\geq
 \frac{c_1\left|x_0-y\right|^N\delta(x)\delta(y)}{c_2\delta(x_0)\delta(y)\left|x-y\right|^N}.
$$
 Whence, letting $y$ tend to $z$ we obtain that
$$
K_D(x,z)\geq c_3\frac{\delta(x)}{\left|x-z\right|^N}
$$
where $c_3$ is a positive constant not depending on $x$ and $z$.
This and formula~(\ref{mar}) yield~(\ref{mar1}). On the other hand,
in~\cite{zha86} it is shown that
 there exists $c_4>0$ such that for every $x,y\in
 D$,
$$
G_D(x,y)\leq c_4\frac{\delta(x)}{\left|x-y\right|^{N-1}}.
$$
Hence, using~(\ref{mar1}) it follows that
\begin{displaymath}
\begin{split}
\sup_{x\in D}\frac{1}{h(x)}\int_DG_D(x,y)dy
&=\frac{c_4}{C} \sup_{x\in
D}\int_D\frac{dy}{\left|x-y\right|^{N-1}}<\infty.
\end{split}
\end{displaymath}
\end{proof}

\begin{thm}
 Assume that
\begin{equation}\label{sc1}
 \lim_{t\rightarrow \infty}\frac{\varphi(t)}{t}=0.
 \end{equation}
 Then for every $f \in \cc{C}^+(\partial D),$ there exists a positive constant $\alpha_f$ such that
    $H_D^\varphi (\alpha f)\approx H_D(\alpha f)$ for every  $\alpha\geq
    \alpha_f$.
    \end{thm}
\begin{proof}
Let  $f\in\cc{C}^+(\partial D)$ be  non trivial and let $h=H_D f$.
By the previous lemma, there exists $c>0$ (depending on $h$) such
that for every $\alpha>0$ and every $x\in D$,
$$
G_D\varphi(\alpha h)(x)\leq \varphi(\alpha \|h\|) G_D 1(x)\leq c
\varphi(\alpha \|h\|) h(x).
$$
Therefore
$$
\sup_{x\in D}\frac{G_D\varphi(\alpha h)\left(x\right)}{\alpha
h\left(x\right)}\leq c\frac{\varphi(\alpha \|h\|)}{\alpha}.
$$
On the other hand, by~(\ref{sc1})  there exists $A>0$ such that for
every $t\geq A$,
$$
\frac{\varphi(t)}{t}< \frac{1}{c \|h\|}.
$$
Take $\alpha_f:=A/\|h\|$. Then for every $\alpha\geq \alpha_f$, we
have
$$
\sup_{x\in D}\frac{G_D\varphi(\alpha h)\left(x\right)}{\alpha
h\left(x\right)}<1,
$$
which yields, by Proposition~\ref{cspro}, that $H_D^\varphi (\alpha
f)\approx H_D(\alpha f)$.
\end{proof}

\section{More about problem (1)}

This last section is devoted to investigate the problem (1) in the
case where $\varphi$ is nonincreasing. Let us notice  that, in this
setting, we do not guarantee the existence nor the uniqueness of the
solution to problem~(1) and hence the operator $H_D^\varphi$ is no
longer defined. As above, we assume that~$D$ is a $C^{1,1}$ bounded
domain of $\mathbb{R}^N,$ $N\geq 3.$
\begin{thm}
Let $\varphi:\mathbb{R}_+\to \mathbb{R}_+$ be a continuous
nonincreasing function, $f\in\cc{C}^+(\partial D)$ and let $h=H_Df$
such that
\begin{equation}\label{sc2}
\sup_{x\in D}E_h^x\bint
\int_0^{\tau_D}\frac{1}{h(X_s)}ds\eint \leq \frac{1}{e\,\varphi(0)}.
\end{equation}
Then the  problem (\ref{e31}) possesses  a solution $u\in
\mathcal{C}^+(\overline{D})$ satisfying $u\approx h.$
\end{thm}
\begin{proof} Of course  we assume that $\varphi(0)>0$ and $f$ is non trivial. By hypothesis,
$$
c:=\sup_{x\in D}\frac{1}{h(x)}\int_D G_D(x,y)\,dy \leq
\frac{1}{e\,\varphi(0)}.
$$
It is easily seen that there exists $b>0$ such that $e^b\varphi(0)c=
b.$ We observe that the set
$$
\Lambda =\left\{u\in \cc{C}(\overline{D}); e^{-b}h\leq u\leq h\right\}
$$
is  closed bounded and convex in $\cc{C}(\overline{D}).$ Consider
$T:\Lambda\rightarrow \cc{C}(\overline{D})$ defined by
$$
Tu(x)=E^x\left[h(X_{\tau_D})\exp\left(-\int_{0}^{\tau_D}\frac{\varphi(u(X_s))}{u(X_s)}ds\right)\right],\quad\
x\in D.
$$
Then, it is clear that $Tu\leq h$ for every $u\in \Lambda$.
Furthermore, for every $x\in D,$
\begin{displaymath}
\begin{split}
\frac{Tu(x)}{h(x)}&=E_h^x\left[\exp\left(-\int_{0}^{\tau_{D}}\frac{\varphi(u(X_s))}{u(X_s)}ds\right)\right]\\
&\geq E_h^x\left[\exp\left(-e^{b}\varphi(0)\int_{0}^{\tau_D}\frac{1}{h(X_s)}ds\right)\right]\\
&\geq \exp\left(-e^b\varphi(0)E^x_h\left[\int_{0}^{\tau_D}\frac{1}{h(X_s)}ds\right]\right)\\
&\geq \exp\left(-e^b\varphi(0)c\right)\\
&= \exp(-b).
\end{split}
\end{displaymath}
This yields that $T(\Lambda)\subset\Lambda.$ On the other hand, for
every $u\in\Lambda$ we have
 $$e^{-b}\frac{\varphi(u)}{u}Tu\leq \varphi(0).$$
 So, in virtue of  Lemma~\ref{l221}, we deduce that the family
  $$\left\{\int_DG_D(\cdot ,y)\frac{\varphi(u(y))}{u(y)}Tu(y)dy; u\in\Lambda\right\}$$
is relatively compact  in  $\cc{C}(\overline{D}).$ Seeing that
$$
Tu(x)+\int_D G_D(x , y)\frac{\varphi(u(y))}{u(y)}Tu(y)dy=h(x),\quad
x\in D,$$ we conclude that $T(\Lambda)$ is relatively compact in
$\cc{C}(\overline{D})$ and consequently  $T$ is continuous. Then by
Schauder's fixed point theorem, there exists $u\in \Lambda$ such
that
$$u(x)+\int_D G_D(x,y)\varphi(u(y))dy=h(x),\quad \ x\in D.$$
Hence, $u$ is a solution to the problem (1). Moreover $u\approx h$
since $u\in\Lambda$.
\end{proof}

\begin{cor}
Let $\varphi:\mathbb{R}_+\to \mathbb{R}_+$ be a continuous
nonincreasing function. For every  $f\in\cc{C}^+(\partial D),$ there
exists $\alpha_f>0$ such that for every $\alpha\geq \alpha_f$ the
problem
\begin{equation*}
\left\{\begin{array}{rlll}\Delta u&=&2\varphi(u)& \mbox{in}\ D,\\
u&=&\alpha f&\mbox{on}\ \partial D,\end{array}\right.
\end{equation*}
admits a solution
 $u\in \mathcal{C}^+(\overline{D})$ satisfying $u\approx H_D f.$
\end{cor}
\begin{proof} Let $f\in\cc{C}^+(\partial D)$ be non trivial and let
$h=H_Df$. It suffices to consider
$$
\alpha_f=e\varphi(0)\sup_{x\in D}\frac{1}{h(x)}\int_D G_D(x,y)\,dy
$$
and apply the previous theorem for $\alpha f$, $\alpha\geq
\alpha_f$.
\end{proof}


\begin{thebibliography}{99}

\bibitem{2} R. Atar,  S. Athreya, Z.Q. Chen,  \emph{Exit time, Green function and semilinear elliptic equations.}
Electron. J. Probab. 14 (2009), no. 3, 50--71.

\bibitem{3}  S. Athreya, \emph{On a singular semilinear elliptic boundary value problem and the boundary Harnack principle.}
 Potential Anal. 17 (2002), no. 3, 293--301.

\bibitem{baha02}  A. Baalal, W.  Hansen, \emph{Nonlinear perturbation of balayage
spaces.} Ann. Acad. Sci. Fenn. Math. 27 (2002), no.~1, 163--172.

\bibitem{12} R.F. Bass,  \emph{Probabilistic techniques in analysis.} Springer-Verlag, New York  (1995).

\bibitem{4}  Z.Q. Chen, R.J.  Williams, Z. Zhao,
\emph{On the existence of positive solutions of semilinear elliptic
equations with Dirichlet boundary conditions.}
 Math. Ann. 298 (1994), no. 3, 543--556.

\bibitem{5} K.L. Chung, Z. Zhao,  \emph{From Brownian motion to Schrodinger's equation. }
Springer-Verlag, Berlin (2001).

\bibitem{6} J.L. Doob,  \emph{Classical potential theory and its probabilistic counterpart}.
Springer-Verlag, New York (1984).

\bibitem{dyn00}  E.B. Dynkin,  \emph{Solutions of semilinear differential equations related to
              harmonic functions}. J. Funct. Anal. 170 (2000), no.~2,
              464--474.

\bibitem{elm} K. El Mabrouk,
 \emph{Semilinear perturbations of harmonic spaces, {L}iouville property and a boundary value problem.}
 Potential Anal. 19 (2003), no.~1, 35--50.

\bibitem{8} K. El Mabrouk,  \emph{Positive solutions to singular semilinear elliptic problems.}
Positivity 10 (2006), no. 4, 665--580.

\bibitem{10} S.C. Port, C.J.  Stone, \emph{Brownian motion and classical potential theory}. Academic Press, New York-London, 1978.

\bibitem{rud}  W. Rudin,  \emph{Real and complex analysis}. Second edition. McGraw-Hill
Book Co.,  New York-D\"usseldorf-Johannesburg (1974).

\bibitem{zha86}  Z. Zhao,  \emph{Green function for {S}chr\"odinger operator and conditioned {F}eynman-{K}ac
gauge}. J. Math. Anal. Appl. 116 (1986), no.~2, 309--334.


\end{thebibliography}
\end{document}